\documentclass[10pt]{article}
\usepackage[dvips]{graphicx}
\usepackage{amssymb,color}
\usepackage[latin1]{inputenc}
\usepackage{epic}
\usepackage{pstricks}
\usepackage{pst-node}
\usepackage{pstricks-add}
\usepackage[latin1]{inputenc}
\usepackage{color}
\usepackage{lettrine}
\usepackage{calc,pifont}
\usepackage[noindentafter,calcwidth]{titlesec}
\usepackage{pstricks}
\usepackage{pst-plot}
\usepackage{pst-node}
\usepackage{amsmath}
\usepackage{amssymb}
\usepackage{amsthm}
\usepackage{subfigure}
\usepackage{lineno}

\def\BBox{\kern  -0.2cm\hbox{\vrule width 0.2cm height 0.2cm}}

\newtheorem{example}{Example}

\newtheorem{teo}{Theorem}[section]

\newtheorem{conje}{Conjecture}[section]
\newtheorem{coro}[teo]{Corollary}

\theoremstyle{definition}

\theoremstyle{remark}

\title{On transversal numbers of intersecting straight line systems and intersecting segment systems}
\author{Adrián Vázquez Ávila\thanks{adrian.vazquez@unaq.edu.mx}\\
	{\small Subdirección de Ingeniería y Posgrado}\\
	{\small Universidad Aeronáutica en Querétaro}\\
}

\date{}

\begin{document}
	\maketitle
	\begin{abstract}
	An intersecting $r$-uniform straight line system is an intersecting linear system whose lines consist of $r$ points on straight line segment of $\mathbb{R}^2$ and any two lines share a point. 
	
	Recently, the author [A. V{\' a}zquez-{\' A}vila, \emph{On intersecting straight line systems}, J. Discret. Math. Sci. Cryptogr. Accepted] proved that any intersecting $r$-uniform straight line system $(P,\mathcal{L})$  has transversal number at most $\nu_2-1$, with $r\geq\nu_2$, where $\nu_2$ is the maximum cardinality of a subset of lines $R\subseteq\mathcal{L}$ such that every triplet of different elements of $R$ do not have a common point. 
	
	In this paper, we improve such upper bound if the intersecting $r$-uniform straight line system satisfies $r=\nu_2$. This result has immediate consequences for some questions given by Oliveros et al. [D. Oliveros, C. O'Neill and S. Zerbib, \emph{The geometry and combinatorics of discrete line segment hypergraphs}, Discrete Math. {\bf 343} (2020), no. 6, 111825].  
\end{abstract}
	
	
	{\bf Math. Subj. Class.:} ~05C65, 05C69.
	\section{Introduction}
	A \emph{linear system} is a pair $(P,\mathcal{L})$ where
	$\mathcal{L}$ is a family of subsets on a ground finite set $P$, such that $|l\cap l^\prime|\leq 1$, for every pair of distinct subsets $l,l^\prime \in \mathcal{L}$. The linear system $(P,\mathcal{L})$ is \emph{intersecting} if $|l\cap l^\prime|=1$, for every pair of distinct subsets $l,l^\prime \in \mathcal{L}$. The elements of $P$ and $\mathcal{L}$ are called \emph{points} and \emph{lines}, respectively. A line with exactly $r$ points is called a \emph{$r$-line}; when all the lines of $\mathcal{L}$ are $r$-lines the linear system is called  \emph{$r$-uniform}. In this context, a \emph{simple graph} is an 2-uniform linear system.

	Let $(P^{\prime},\mathcal{L}^{\prime })$ and $(P,\mathcal{L})$ be two linear systems. The linear systems $(P^{\prime},\mathcal{L}^{\prime })$ and $(P,\mathcal{L})$ are \emph{isomorphic}, $(P^{\prime },\mathcal{L}^{\prime })\simeq(P,\mathcal{L})$, if after of deleting points of degree 1 or 0 from both, the linear systems $(P^{\prime},\mathcal{L}^{\prime })$ and $(P,\mathcal{L})$ are isomorphic as hypergraphs, see \cite{MR3727901}.
	
	A subset $T\subseteq P$ is a \emph{transversal} (also called \emph{vertex cover} or \emph{hitting set} in many papers, as example \cite{MR1167472,Dorfling,MR2038482,MR1157424, 	MR3373359,MR2383447,MR2765536,MR1921545,MR1149871,MR1185788,Stersou}) of $(P,\mathcal{L})$ if every line $l\in\mathcal{L}$ satisfies $T\cap l\neq\emptyset$. The \emph{transversal number} of $(P,\mathcal{L})$, $\tau=\tau(P,\mathcal{L})$, is the smallest possible cardinality of a
	transversal of $(P,\mathcal{L})$. On the other hand, a subset of lines $R$ of $\mathcal{L}$ is called a \emph{2-packing} of $(P,\mathcal{L})$ if every triplet of different elements of $R$ do not have a common point. The maximum cardinality of a 2-packing of $(P,\mathcal{L})$, $\nu_2=\nu_2(P,\mathcal{L})$, is called \emph{2-packing number} of $(P,\mathcal{L})$. There are some papers which have been studied this parameter, see for example \cite{AvilaLetters,AvilaAKCE,Avila_covering_grphs,MR3727901,AvilaEnotes,Avila_dom_gra,AvilaArscomb,AvilaJDMSC}. 
	
	 An interesting problem is found an upper bound of $\tau$ in terms of a function of $\nu_2$ for any linear system. Araujo-Pardo et al. in \cite{MR3727901} proved a relationship between the transversal and the 2-packing numbers for linear systems 
	\begin{equation}\label{desigualdad}
	\lceil \nu_{2}/2\rceil\leq\tau
	\leq \frac{\nu_2(\nu_2-1)}{2}.
	\end{equation}
	
	The transversal number of any linear system is upper bounded by a quadratic function of their 2-packing number. For some linear systems, the transversal number is bounded above by a linear function of their 2-packing number, see \cite{AvilaLetters,AvilaAKCE,Avila_covering_grphs,MR3727901,AvilaJDMSC}. However, Eustis and Verstra{\"e}te in \cite{MR3021333} proved, using probabilistic methods, the existence of $k$-uniform linear systems $(P,\mathcal{L})$ for
	infinitely many $k$'s and $n=|P|$ large enough, which transversal number is $\tau=n-o(n)$. This $k$-uniform linear
	systems has 2-packing number upper bounded by $\frac{2n}{k}$.

	Through this paper, every linear systems $(P,\mathcal{L})$ are considered with $|\mathcal{L}|>\nu_2$ due to the fact $|\mathcal{L}|=\nu_2$ if and only if $\Delta\leq2$. Alfaro and Vázquez-Ávila in \cite{AvilaLetters} proved that any linear system with $\Delta=2$ satisfies $\tau\leq\nu_2-1$.
	
A \emph{straight line representation} on $\mathbb{R}^{2}$ of a
linear system $(P,\mathcal{L})$ maps each point $x\in P$ to a point $p(x)$ of $\mathbb{R}^{2}$, and each line $L\in\mathcal{L}$ to a straight line segment $l(L)$ of $\mathbb{R}^{2}$ in such a way that for each point $x\in P$ and line $L\in\mathcal{L}$ satisfies $p(x)\in l(L)$ if and only if $x\in L$, and for each pair of distinct lines	$L,L^\prime\in\mathcal{L}$ satisfies $l(L)\cap l(L^\prime)=\{p(x):x\in L\cap L^\prime\}$. A
\emph{straight line system} $(P,\mathcal{L})$ is a linear system, such that it has a straight line representation on $\mathbb{R}^{2}$, see \cite{MR3727901}. 
	
	Araujo-Pardo et al. \cite{MR3727901} proved that any linear system with $\nu_2\in\{2,3\}$ satisfies $\tau\leq\nu_2-1$. Also, they proved any linear system with $\nu_2=4$ satisfies $\tau\leq4$; the inequality holds for a certain linear subsystems of the finite projective plane of order 3, $\Pi_3$; and they proved that, this linear subsystem have not a straight line representation on $\mathbb{R}^2$.
	
	\begin{teo}\label{teo:nu_2=4}\cite{MR3727901}
	Let $(P,\mathcal{L})$ be an straight line system. If $\nu_2\in\{2,3,4\}$, then $\tau\leq\nu_2-1$.
	\end{teo}

As a simple consequence of Theorem \ref{teo:nu_2=4} is the following:

\begin{coro}\label{coro:nu_2=4}
Let $(P,\mathcal{L})$ be an intersecting $r$-uniform straight line system with  $r\geq\nu_2$. If $\nu_2\in\{2,3,4\}$, then $\tau\leq\nu_2-1$.	
\end{coro}

Recently, the author \cite{AvilaJDMSC} proved the following:

\begin{teo}\cite{AvilaJDMSC}\label{teo_anterior}
Any intersecting $r$-uniform straight line system satisfies $\tau\leq \nu_2-1$, where $r\geq\nu_2$.	
\end{teo}

In this paper, we improve the upper bound given in Theorem \ref{teo_anterior} when $r=\nu_2$.

\begin{teo}
Any intersecting $\nu_2$-uniform straight line system with $r\geq\nu_2$ satisfies $\tau\leq\frac{2}{3}(\nu_2+1)$.	
\end{teo}

And we present some others interesting results. Also, we present some consequences for certain intersecting subsystems of uniform straight line systems called \emph{segment systems}.

\section{Main results}

In this section we present the main results of the paper. Before to this, we need some definitions and results.

According to \cite{Kaufmann:2009:SDH:1506879.1506920}, any straight line system is Zykov-planar, see also \cite{MR0401556}. Zykov proposed to represent the lines of a set system by a subset of the faces of a planar map on $\mathbb{R}^2$, i.e., a set system $(X,\mathcal{F})$ is Zykov-planar if there exists a planar graph $G$ (not necessarily a simple graph) such that $V(G)=X$ and $G$ can be drawn in the plane with faces of $G$ two-colored (say red and blue) so that there exists a bijection between the red faces of $G$ and the subsets of $\mathcal{F}$ such that a point $x$ is incident with a red face if and only if it is incident with the corresponding subset. Walsh in \cite{MR0360328} shown the Zykov's definition is equivalent to the following: A set system $(X,\mathcal{F})$ is Zykov-planar if and only if the Levi graph $B(X,\mathcal{F})$ is planar. It is well known for any planar graph $G$ the size of $G$, $|E(G)|$,  is upper bounded by $\frac{\kappa(|V(G)|-2)}{\kappa-2}$ (see for example \cite{Bondy}), where $\kappa$ is the \emph{girth} of $G$ (the length of a shortest cycle contained in the graph $G$).

The \emph{Levi graph} of a linear system $(P,\mathcal{L})$, $B(P,\mathcal{L})$, is a bipartite graph with vertex set
$V=P\cup\mathcal{L}$, where two vertices $p\in P$, and $L\in\mathcal{L}$ are adjacent if and only if $p\in L$.

\begin{teo}\label{teo:nu_uniform}
If $(P,\mathcal{L})$ is an intersecting $\nu_2$-uniform straight line system with $\nu_2\geq3$, then $|\mathcal{L}|\leq\frac{3\nu_2+1}{2}$.	
\end{teo}
\begin{proof}
	Let $B=B(P,\mathcal{L})$ be the Levi graph of $(P,\mathcal{L})$. Since the linear system is a straight line system, then the Levi graph $B$ is a planar graph. It is not hard to see the size of the girth of $B$ is $6$ (since $\nu_2\geq3$). We have $|P|\leq \nu_2^2-\nu_2+1$. Hence $$\frac{\kappa(|V(B)|-2)}{\kappa-2}=\frac{3}{2}(|P|+|\mathcal{L}|-2)=\frac{3}{2}(\nu_2^2-\nu_2+|\mathcal{L}|-1).$$ On the other hand, since every line has $\nu_2$ points, then $|E(B)|=\nu_2|\mathcal{L}|$. Since the Levi graph is planar, then 
	$\nu_2|\mathcal{L}|\leq\frac{3}{2}(|P|+|\mathcal{L}|-2)$, which implies that
	\begin{equation*}
	|\mathcal{L}|\leq\frac{3(\nu_2^2-\nu_2-1)}{2\nu_2-3}\leq\frac{3\nu_2+1}{2}.
	\end{equation*}	
\end{proof}

Let $(P_{\nu_2},\mathcal{L}_{\nu_2})$ be a $\nu_2$-uniform linear system. The $r$-uniform linear system $(P_r,\mathcal{L}_r)$, with $r+1\geq\nu_2$, is \emph{$\nu_2$-isomorphic} to $(P_{\nu_2},\mathcal{L}_{\nu_2})$, $(P_r,\mathcal{L}_r)\simeq_{\nu_2}(P_{\nu_2},\mathcal{L}_{\nu_2})$, if after of either adding or deleting some points of degree 1 from $(P_r,\mathcal{L}_r)$ the resultant linear systems is isomorphic to $(P_{\nu_2},\mathcal{L}_{\nu_2})$. Notice that, if $(P_r,\mathcal{L}_r)\simeq_{\nu_2}(P_{\nu_2},\mathcal{L}_{\nu_2})$, then (trivially) $\nu_2(P_r,\mathcal{L}_r)=\nu_2$.

\begin{coro}\label{coro:iso_nu_2}
Let $(P_{\nu_2},\mathcal{L}_{\nu_2})$ be an intersecting $\nu_2$-uniform straight line system with $\nu_2\geq3$. If $(P_r,\mathcal{L}_r)\simeq_{\nu_2}(P_{\nu_2},\mathcal{L}_{\nu_2})$, then $|\mathcal{L}_r|=|\mathcal{L}_{\nu_2}|\leq\frac{3\nu_2+1}{2}$.	
\end{coro}

\begin{teo}\label{teo:main_2}
If $(P,\mathcal{L})$ is an intersecting $\nu_2$-uniform straight line system with $\nu_2\geq3$, then $\tau\leq\frac{2}{3}(\nu_2+1)$.	
\end{teo}
\begin{proof}
	Suppose that $\nu_2\geq2$ is an even integer. By Theorem \ref{teo:nu_uniform}
	\begin{equation*}
	|\mathcal{L}|\leq\frac{3}{2}\nu_2=\nu_2+\frac{\nu_2}{2}.
	\end{equation*}
	Suppose that $\frac{\nu_2}{2}=3l+i$, for some $l\in\mathbb{Z}$ and $i\in\{0,1,2\}$. Analyzing all cases for $i$, it can be proved that $\tau\leq\frac{2\nu_2+1}{3}$ (see case $i=2$).
	
	On the other hand, suppose that $\nu_2\geq3$ is an odd integer. By Theorem \ref{teo:nu_uniform}
	\begin{equation*}
	|\mathcal{L}|\leq\nu_2+\frac{\nu_2+1}{2}.
	\end{equation*}
	Suppose that $\frac{\nu_2+1}{2}=3l+i$, for some $l\in\mathbb{Z}$ and $i\in\{0,1,2\}$. Analyzing all cases for $i$, it can be proved that $\tau\leq\frac{2}{3}(\nu_2+1)$ (see case $i=1$).
\end{proof}

\begin{coro}\label{coro:main}
Let $(P,\mathcal{L})$ be an intersecting $\nu_2$-uniform straight line system with $|\mathcal{L}|>\nu_2$ and $\nu_2\geq3$ an integer. If $\Delta\geq\lfloor\frac{\nu_2}{2}\rfloor+2$ then $\tau=\lceil\frac{\nu_2}{2}\rceil$. 
\end{coro}
\begin{proof}
Suppose that $\nu_2\geq4$ is an even integer. By Theorem \ref{teo:nu_uniform}
\begin{equation*}
|\mathcal{L}|\leq\nu_2+\frac{\nu_2}{2}=\nu_2-2+\frac{\nu_2}{2}+2.
\end{equation*}
Hence
\begin{equation*}
\tau\leq\frac{\nu_2-2}{2}+\frac{\frac{\nu_2}{2}+2}{\Delta}\leq\frac{\nu_2}{2}=\lceil\frac{\nu_2}{2}\rceil.
\end{equation*}
On the other hand, suppose that $\nu_2\geq3$ is an odd integer. By Theorem \ref{teo:nu_uniform}
\begin{equation*}
|\mathcal{L}|\leq\nu_2+\frac{\nu_2+1}{2}=\nu_2-1+\frac{\nu_2-1}{2}+2.
\end{equation*}
Hence
\begin{equation*}
\tau\leq\frac{\nu_2-1}{2}+\frac{\frac{\nu_2-1}{2}+2}{\Delta}\leq\frac{\nu_2+1}{2}=\lceil\frac{\nu_2}{2}\rceil.
\end{equation*}
Therefore, by Equation (\ref{desigualdad}), we have
\begin{equation*}
\lceil\frac{\nu_2}{2}\rceil\leq\tau\leq\lceil\frac{\nu_2}{2}\rceil,
\end{equation*} 
which implies $\tau=\lceil\frac{\nu_2}{2}\rceil$.
\end{proof}

\section{$r$-segment systems} 
Let $r\geq2$ be an integer. A linear system $(P,\mathcal{L})$, where $P\subseteq\mathbb{Z}^2$, is an \emph{$r$-segment system} if every line in $\mathcal{L}$ consists of $r$ consecutive integer points on some line in $\mathbb{R}^2$, and every line in $\mathbb{R}^2$ contains at most one line of $\mathcal{L}$, see \cite{Oliveros}. It is easy to see that any $r$-segment system is an $r$-uniform straight line system, see \cite{Oliveros}. An immediate consequence of this definition is the following:

\begin{coro}\label{coro:segment_subsystem}
Let $(P,\mathcal{L})$ be an intersecting $r$-segment system and $l\in\mathcal{L}$ any line. If $(P',\mathcal{L}')$ is the linear subsystem induced by $\mathcal{L}\setminus\{l\}$ taking away the points of degree 0, then $(P',\mathcal{L}')$ is an intersecting $r$-segment system. 
\end{coro}  

Oliveros et al. \cite{Oliveros} proved the following (which is a simple consequence of Corollary \ref{coro:main}):
\begin{teo}\cite{Oliveros}\label{teo:oliveros}
Any intersecting $r$-segment system with $r\geq3$ have a point of degree 1. Consequently, $\tau\leq r-1$.
\end{teo}

As a simple consequence of Theorem \ref{teo:oliveros} and Corollary \ref{coro:segment_subsystem} is the following:

\begin{coro}
Let $(P,\mathcal{L})$ be an intersecting $r$-uniform linear system with $\nu_2\geq3$. If $r=\nu_2-1$, then $(P,\mathcal{L})$ can be isomorphic to an intersecting $r$-segment system.
\end{coro}

\begin{coro}
If $(P,\mathcal{L})$ is an intersecting $r$-segment system with $\nu_2\geq3$, then $r\geq\nu_2$.
\end{coro}

As a simple consequence of Theorem \ref{teo:nu_uniform} and Theorem \ref{teo:main_2} are the following:

\begin{coro}\label{coro:lines_segment}
If $(P,\mathcal{L})$ is an intersecting $\nu_2$-segment system with $r\geq\nu_2\geq3$, then $|\mathcal{L}|\leq\frac{3\nu_2+1}{2}$.	
\end{coro}

\begin{coro}\label{coro:tau_segment}
	If $(P,\mathcal{L})$ is an intersecting $\nu_2$-segment system with $\nu_2\geq3$, then $\tau\leq\frac{2}{3}(\nu_2+1)$.
\end{coro}

\begin{example}
The linear system depicted in Figure \ref{fig:Fano_pecas} (a) is an intersecting 4-uniform straight line system with $\nu_2=4$ (red lines); moreover this linear system is isomorphic to an intersecting 4-segment system (see Figure \ref{fig:Fano_pecas} (b)) which satisfies the transversal number obtained from Corollary \ref{coro:tau_segment}. According to \cite{Oliveros}, this is the intersecting 4-segment system with the maximum number of lines (6 lines). This number can be obtained from Corollary \ref{coro:lines_segment}. 
\end{example}

Let $(P,\mathcal{L})$ an intersecting $r$-segment system with $r\geq\nu_2\geq3$ be an integer. A \emph{triangle} $T$ of $(P,\mathcal{L})$, is obtained by three points of $P$ non collinear and the three lines induced by them such that any two points of $T$ are consecutive integer points, see \cite{Oliveros}. 

Oliveros et al. \cite{Oliveros} proved that any intersecting 5-segment system satisfies $\tau\leq3$ (see Theorem 3.7 in \cite{Oliveros}). Also, they proved that any intersecting 5-segment has at most 6 lines if it contains a triangle; and moreover, those linear systems satisfy $\nu_2=4$. Hence, if $(P_5,\mathcal{L}_5)$ is any intersecting 5-segment system containing a triangle, then $(P_5,\mathcal{L}_5)\simeq_4(P_4,\mathcal{L}_4)$ (see Theorem 3.7 in \cite{Oliveros}). Therefore, $|\mathcal{L}_5|=|\mathcal{L}_4|\leq6$, by Corollary \ref{coro:iso_nu_2}.

Hence, Corollary \ref{coro:lines_segment}, Corollary \ref{coro:iso_nu_2} and Theorem \ref{teo:nu_uniform} gives a partial answer to the Question 3.9 given in \cite{Oliveros}:

{\bf Question 3.9.} What is the maximum number of lines that can occur in an intersecting $r$-segment system containing a triangle?.

Also, Oliveros et al. in \cite{Oliveros} conjectured the following:

\begin{conje}\label{conje:Oliveros}
Any intersecting $r$-segment system satisfies $\tau\leq\lceil r/2\rceil$, for $r\geq5$ an integer.
\end{conje}

\begin{figure}
	\begin{center}
		\subfigure[]{\includegraphics[height =3.5cm]{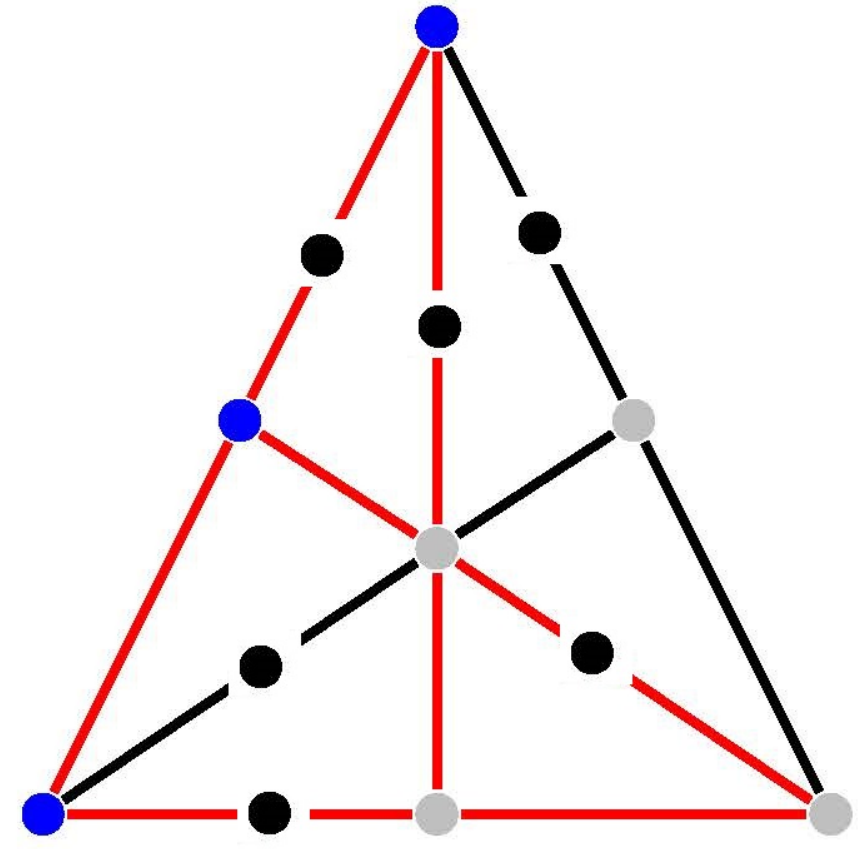}}
		\hspace{1cm}
		\subfigure[]{\includegraphics[height =3.5cm]{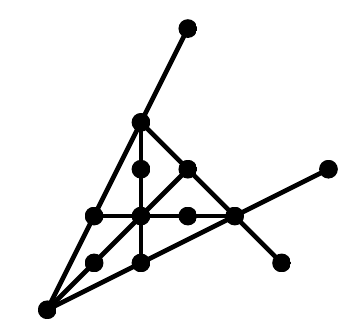}}
	\end{center}
	\caption{Projective plane of order 2 (Fano's plane) without a line (any) with adding a single point for each one line and (b) is the corresponding isomorphic intersecting 4-segment system (Fig. 2. (b) from \cite{Oliveros}).}\label{fig:Fano_pecas}
\end{figure}

Hence, if Conjecture \ref{conje:Oliveros} is true, then by Equation (\ref{desigualdad}), we have $$\lceil\nu_2/2\rceil\leq\tau\leq\lceil r/2\rceil,$$and the equality holds if and only if either $\nu_2=r$ or $\nu_2=r-1$ if $r$ is an even integer.

Alfaro and Vázquez-Ávila in \cite{AvilaLetters} proved that any intersecting linear system with $\Delta=2$ satisfies $\tau=\lceil\nu_2/2\rceil$. Oliveros et al. \cite{Oliveros} constructed an $r$-segment system with $\tau=\lceil r/2\rceil$, for every $r\geq3$ an integer. Those $r$-segment systems are such that $\nu_2=r$ and $\Delta=2$. 

Finally, as a simple consequence of Corollary \ref{coro:main} is the following:

\begin{coro}\label{coro:final}
Let $(P,\mathcal{L})$ be an intersecting $\nu_2$-segment system with  $\nu_2\geq3$ an integer. If $\Delta\geq\lfloor\frac{\nu_2}{2}\rfloor+2$ then $\tau=\lceil\frac{\nu_2}{2}\rceil$. 
\end{coro}

Corollary \ref{coro:final} shows examples of intersecting segment systems satisfying the Conjecture \ref{conje:Oliveros}. 

{\bf Acknowledgment}

Research was partially supported by SNI and CONACyT. 

\bibliographystyle{amsplain}

\end{document}